\documentclass[12pt]{article}
\usepackage{amsthm, amsfonts, amsmath, amssymb}

\usepackage{enumerate, enumitem}
\newtheorem{theorem}{Theorem}[section]
\newtheorem{lemma}[theorem]{Lemma}

\newtheorem{definition}[theorem]{Definition}

\usepackage[mode=buildnew]{standalone}

\usepackage{xcolor}
    \definecolor{custom_red}{RGB}{150, 50, 50}
    \definecolor{custom_green}{RGB}{50, 150, 50}
    \definecolor{custom_blue}{RGB}{50, 50, 150}
    \definecolor{custom_cyan}{RGB}{0, 174, 239}
    \definecolor{custom_black}{RGB}{0, 0, 0}
    \definecolor{custom_purple}{RGB}{200, 150, 250}
    \definecolor{custom_textcolor}{RGB}{0, 0, 0}
\usepackage{tikz}

\title{    Adiabatic Invariant Actions for    \\Partially Integrable Systems
}
\author{*Amir Khodaeian, **Konstantinos Kourliouros, ***Dmitry Turaev
\\[1ex]
Department of Mathematics, Imperial College, London
\\[1ex]
*a.khodaeian-karim22@imperial.ac.uk
\\[1ex]
**k.kourliouros@googlemail.com
\\[1ex]
***d.turaev@imperial.ac.uk
}

\date{\today}

\begin{document}

\maketitle
\begin{abstract}
    We introduce action integrals for partially integrable Hamiltonian systems that interpolate between the non-integrable and completely integrable theory. We show that under the assumption of ergodicity of the Hamiltonian system on the common level sets of its integrals, these actions become constants of motion for the averaged dynamics and, consequently, adiabatic invariants when the Hamiltonian function is allowed to slowly change with time. 
\end{abstract}
\section{Introduction}

We consider a system of Hamiltonian differential equations with a Hamiltonian function $H$ that {\em depends slowly on time}:
    \begin{align} 
        \frac{dq}{dt} &= \partial_{p}H(q,p;\theta),        \quad
        \frac{dp}{dt} = -\;\partial_{q} H(q,p;\theta), \label{Hamiltonian_ODE}
        \\[1ex]
        &
        \frac{d\theta}{dt} = \varepsilon, \label{thetaeq}
    \end{align}
where $q\in \mathbb{R}^{n}$, $p\in \mathbb{R}^{n}$, $\theta\in \mathbb{R}^1$, the function $H$ is $C^\infty$ smooth, and $\varepsilon$ is assumed to be a small parameter. The question of the long-time behaviour of such systems lies at the origins of the theory of Hamiltonian dynamics, and of the foundations of classical and quantum mechanics, thermodynamics and statistical physics. 

The short-time behaviour (meaning on a finite interval of time $t$ in the limit $\varepsilon\to 0$) is determined by the {\em frozen system} in the $y$-space ($\mathbb{R}^{2n}$ with the standard symplectic form), given by (\ref{Hamiltonian_ODE}) with the slow variable $\theta$ taken to be constant. Throughout the paper we will consider the case where the frozen system is {\em partially integrable}, i.e. such that for each $\theta$, the number of independent single-valued first integrals of (\ref{Hamiltonian_ODE}) that stand in involution is strictly less than the number $n$ of degrees of freedom of the system. The main examples of such Hamiltonian systems are given by systems with continuous symmetries (e.g. systems with an $SO(2)$ symmetry preserve angular momentum) and systems that consist of non-interacting subsystems (then the energy of each subsystem is an integral); there are many other non-trivial examples. 

Recall that an integral for the frozen system is a smooth function $Q:\mathbb{R}^{2n}\rightarrow\mathbb{R}$ which is preserved by the Hamiltonian flow, or equivalently, it stands in involution with $H$. This means that $\{Q, H\} = 0$ where
$
    \{Q, H\}
    :=
    \sum_{i = 1}^n \left(   
        \partial_{q_i}Q\ \partial_{p_i}H - \partial_{p_i}Q\ \partial_{q_i}
    \right)
$
is the Poisson bracket associated to the standard symplectic form
$
    \omega=\sum_{i=1}^ndp_i\wedge dq_i
$
in $\mathbb{R}^{2n}$. By definition, the energy $H$ is such an integral, and we assume that there are others: we let the frozen system have a set of $k \le n$ independent integrals $\mathbf{H} := (H_1 = H,\ H_2, \dots, H_k)$ for all $\theta$, defined in a bounded (continuously dependent on $\theta$) region of $\mathbb{R}^{2n}$. The common level of the integrals is the set $M_{\mathbf{h},\theta}=\{y\in \mathbb{R}^{2n}:\mathbf{H}(y; \theta)=\mathbf{h}\}$. The independence of the integrals implies that the common levels are smooth $(2n-k)$-dimensional manifolds, which we also assume to be compact. We assume that the integrals $H_i$ stand in involution, $\{H_i,H_j\}=0$ for all $i,j=1,\dots,k$; it follows that the corresponding Hamiltonian vector fields $X_{H_i}$, $i=1,\dots,k$, restricted onto a manifold $M_{\mathbf{h},\theta}$, span a smooth $k$-dimensional integrable plane field.

At $\varepsilon\neq 0$, the functions $H_i$ are no longer preserved by the full system (\ref{Hamiltonian_ODE}), (\ref{thetaeq}) and become {\em slow variables} (along with the variable $\theta$). Their time evolution is described by the {\em slow subsystem}
\begin{equation}\label{slowsub}
\begin{array}{l}\displaystyle
\frac{d \mathbf{H} }{dt} = \{\mathbf{H}, H\} + \varepsilon \; \partial_\theta \mathbf{H} = \varepsilon \; \partial_\theta \mathbf{H},\\ 
\displaystyle
\qquad\frac{d \theta}{dt}=\varepsilon,
\end{array}
\end{equation}
which gives that $\mathbf{H}$ can acquire an order-$1$ increment over a time interval of order $\varepsilon^{-1}$. 

However, there may exist functions of the slow variables which 
are {\em almost preserved} for such long time-intervals. These functions are called {\em adiabatic invariants}. In a sense, for slowly time-dependent Hamiltonian systems the problem of finding adiabatic invariants replaces the problem of integrability, see \cite{ArKoNe}.

The existence of adiabatic invariants is well known for two classical cases: either when the frozen system is {\em ergodic} with respect to the Liouville measure on almost every level of constant energy $H$ and for almost every $\theta$ (this implies, in particular, that the frozen system has no additional integrals, i.e., $k=1$), or, on the other extreme, when the frozen system is {\em completely integrable}. 

In the completely integrable case, the compact common levels of the integrals are $n$-dimensional Lagrangian tori and in a neighbourhood of each torus there exist \cite{Ar} symplectic {\em action-angle} coordinates $(J,\phi)\in \mathbb{R}^n\times (\mathbb{R}^n/\mathbb{Z}^n)$ for which the frozen system reads
$$\dot J = 0, \qquad \dot\phi = \partial_J H.$$
When writing the full system (\ref{Hamiltonian_ODE}), (\ref{thetaeq}) in the action-angle variables, small corrections appear at $\varepsilon \neq 0$:  
$$\dot J = O(\varepsilon), \qquad \dot\phi = \partial_J H + O(\varepsilon), \qquad\; \dot\theta =\varepsilon,$$
which are $1$-periodic in $\phi$. If the twist condition $\det(\partial_{JJ} H)\neq 0$ holds for all $\theta$, then by averaging over the angles $\phi$, one establishes the classical result that
{\em the actions are adiabatic invariants}:
for a Lebesgue-typical initial condition, the value of the actions $J$ remains $O(\varepsilon)$-close to a constant
on time intervals of order $\varepsilon^{-1}$ at least \cite{Ar,ArKoNe}.
     
The actions $J$ are smooth functions of slow variables $h=\mathbf{H}(y;\theta)$ and $\theta$, and are defined as integrals of the Liouville form along the closed loops on the common level corresponding to the given values of $h$ and $\theta$. As the common level is a Lagrangian torus, the actions depend only on the homology class of the integration paths, and any $n$ pairwise non-homologous loops give $n$ independent action variables.

The other known case corresponds to the frozen system being {\em ergodic} with respect to the Liouville measure on almost every level of constant energy $H$, for almost every $\theta$ (this implies, in particular, that the frozen system has no additional integrals, i.e., $k=1$). Then, by Anosov-Kasuga theory \cite{ArKoNe,An,Ka}, {\em the volume $J(h,\theta)$ of the set bounded by the energy level $H(y,\theta)=h$ is an adiabatic invariant} (in a slightly weaker sense than in the completely integrable case, see Theorem \ref{adiabatic_invariant}). This result is particularly important for the foundations of statistical physics, as the logarithm $\ln J$ coincides with the {\em Gibbs volume entropy} of the system and its adiabatic invariance can be interpreted as the preservation of entropy at reversible adiabatic processes, one of the basic principles of thermodynamics. 

In this paper we attempt at building a theory of geometrically defined adiabatic invariants for the general partially-integrable case ($1 < k < n$),
which is intermediary between the two classical cases (see \cite{Lo}) $k = 1$ and $k = n$.
To the best of our knowledge, previous attempts found in literature are not in full generality. For example see Rugh \cite{Ru} where strict symmetry assumptions are assumed.
As the restrictions on the topology of the common levels are not well-understood beyond the completely integrable case, we avoid explicit topological assumptions and, instead, examine the relation between the existence of adiabatic invariants in the full system and the ergodic properties of the frozen system. The main result is as follows.

For a frozen system with $k$ independent integrals, we define action-like variables $J(h,\theta)$, which coincide with the classical actions at $k=n$ and with the Anosov-Kasuga invariant at $k=1$, and prove that {\em if the frozen system restricted to almost every common level of the $k$ integrals is ergodic with respect to the restriction of the Lebesgue measure, then the variables $J$ are adiabatic invariants of the full system}, see Theorem \ref{adiabatic_invariant}.

By Anosov theorem \cite{An}, when the frozen system is ergodic, in order to show that a slow variable is an adiabatic invariant, it is enough to show that it is an integral of the
{\em averaged system}
\begin{equation}\label{averageslow}
\frac{d}{d\theta} h = \int \partial_\theta \mathbf{H}(y;\theta) 
\mu_{_L}(dy),
\end{equation}
obtained from the slow subsystem (\ref{slowsub}) by averaging over the {\em Liouville measure}
\begin{equation}\label{lmmul}
\mu^{_L}_{(h, \theta)}= \frac{1}{\int\delta(\mathbf{h} - \mathbf{H}(y;\theta)) (dy)^{2n}}\; \delta(\mathbf{h} - \mathbf{H}(y;\theta)) (dy)^{2n}
\end{equation}
(the restriction of the Lebesgue volume to the common level $\{\mathbf{H}=\mathbf{h}\}$). So, we prove Theorem \ref{adiabatic_invariant} by showing that
{\em if the frozen system is ergodic on a dense set of common levels, then the actions $J$ are integrals of the averaged system (\ref{averageslow})}, see Theorem \ref{integral_for_the_averaged_system}. 

Quite surprisingly, beyond the classical cases $k=1$ and $k=n$, this result does not hold without the ergodicity assumption which is used in the proof of Theorem \ref{integral_for_the_averaged_system} in an essential way (moreover, one can built non-ergodic counterexamples to the statement of the theorem). This is important, as the averaged system is the object of study in the Hamiltonian dynamics in its own right. 

The physical intuition that necessitates the study of the averaged system is this: In situations where the fast variables are not observable, one needs to derive a closed system of differential equations that governs the evolution of the slow variables only. One does it by somehow averaging the slow subsystem (\ref{slowsub}) over the fast variables at constant values of $\mathbf{H}$ and $\theta$. It is not a priori given which measure on the constant level of the integrals $\mathbf{H}$ should be taken for such averaging, but the Liouville measure $\mu^{_L}$ appears to be the most natural choice. 

Averaging over $\mu^{_L}$ had been introduced systematically by Gibbs, who called it the {\em microcanonical ensemble}, and is widely and successfully used in statistical physics. One of its most basic facts - the so-called fundamental thermodynamic relation - can be rephrased as that the Gibbs volume entropy is the integral of the averaged equation for the slow evolution of energy. This corresponds to our $k=1$ case, and holds without the ergodicity assumption. 

The assumption of the ergodicity of the Liouville measure is typically used only as one of many possible theoretical justifications for the microcanonical ensemble averaging, but not for the study of the averaged slow dynamics. Our Theorem \ref{integral_for_the_averaged_system} show sthat in the partially-integrable case ($k \neq  1,n$) {\em the ergodicity of the frozen system plays also a different role}, not reported before - it influences the properties of the averaged slow system itself (supplies the averaged system with integrals).

Mathematically, this means that there is a certain connection between the statistical properties of a partially integrable Hamiltonian system and the algebraic properties (e.g. integrability) of the corresponding averaged system. It is possible that this connection has a geometric or topological nature, like in the completely integrable case where the existence of the adiabatic invariants is a consequence of the fact that the common levels are Lagrangian tori (so the action-angle variables can be introduced and the averaging over the angle variables can be performed). We do not have a good insight into this, but it seems to be a topic worth exploring, as it can reveal relations between the ergodicity/non-ergodicity of Hamiltonian dynamics and the symplectic geometry/topology of the common levels of the commuting integrals.

\section{Action-like integrals of the averaged slow system.}

We will view system (\ref{Hamiltonian_ODE}) with $\theta=\theta_0+\varepsilon t$
as a time-dependent Hamiltonian system. Denote its flow map from time $t_0$ to time $t$ as $\varphi_{t_0}^t$. These maps preserve the standard symplectic form
$$\omega = \sum_{i=1}^n dp_i \wedge dq_i$$
in $\mathbb{R}^{2n}$. We also consider the Liouville form
$$\alpha = \sum_{i=1}^n p_i dq_i, \qquad d \alpha =\omega.$$

Given a closed $(2s+1)$-dimensional smooth manifold $\Gamma$, we call its
(pre)contact volume a {\em surface action} of $\Gamma$:
$$I(\Gamma) := \int_\Gamma \alpha \wedge \omega^s.$$
Note that since $\iota_{_{X_H}}\omega = - dH$, where $X_H$ is the vector field defined by the Hamiltonian $H$ as in (\ref{Hamiltonian_ODE}), Stokes theorem implies
\begin{equation}\label{igcon}
\frac{d}{dt} I(\varphi_{t_0}^t \Gamma) = \int_{\varphi_{t_0}^t \Gamma} \iota_{_{X_H}} d(\alpha\wedge\omega^s) = - (s+1) \int_{\varphi_{t_0}^t \Gamma} d (H\; \omega^s)=0.
\end{equation}

Let $\mathcal{M}_{h,\theta}$ denote the manifold $\mathbf{H}(y;\theta)=h$ in $\mathbb{R}^{2n}$ (a common level of the integrals $H_1,\dots,H_k$). Since 
$\{H_i,H_j\}=0$, the value of $H_i$ remains constant along the orbits of the Hamiltonian flow of $H_j$ at $\varepsilon=0$, for all $i,j=1,\dots, k$. Therefore,
the vector fields $X_{H_j}$ are all tangent to the common level $\mathcal{M}_{h,\theta}$ at each of its points.

We have $\iota_{_{X_{H_i}}}\omega_{|_{\mathcal{M}(h,\theta)}}  = - {dH_i}|_{_{\mathcal{M}(h,\theta)}} =0$ for all $i=1,\dots, k$, i.e., the symplectic form $\omega$ vanishes at every $p\in \mathcal{M}_{h,\theta}$ on $k$ linearly independent vectors in $T_p\mathcal{M}_{h,\theta}$. As $\dim(\mathcal{M}_{h,\theta})=2n-k$, this implies
$$\omega^{n-k+1}|_{\mathcal{M}(h,\theta)}=0.$$

Thus, if two closed, smooth, codimension-$(k-1)$ submanifolds of $\mathcal{M}(h,\theta)$ belong to the same 
(ambient) cobordism class in $\mathcal{M}(h, \theta)$, e.g. if they are are isotopic, their surface actions are equal:
\begin{equation}\label{iiz}
\int_{\Gamma} \alpha\wedge \omega^{n-k} = \int_{\Gamma^\prime} \alpha\wedge \omega^{n-k}.
\end{equation}
This means that, for each cobordism class $c$ of codimension-$(k - 1)$ closed submanifolds $\Gamma \in \mathcal{M}(h, \theta)$, the surface action $I(\Gamma)$ is a function of $\theta, h$, and $c$ only.

\begin{definition} \label{definition_of_action_variables}
For a $(2n-2k+1)$-dimensional cobordism class $c$ of $\mathcal{M}(h,\theta)$, 
we define the action 
\[ J_c(h,\theta):= I(\Gamma)=\int_{\Gamma} \alpha\wedge \omega^{n-k},
\]
where $\Gamma$ is any closed $(2n-2k+1)$-dimensional submanifold of $\mathcal{M}(h,\theta)$ of class $c$.
\end{definition}

By the definition, for every fixed $c$, the action $J_c(\mathbf{H}(y; \theta), \theta)$ is a slow variable. In the completely integrable case, the number of integrals $k=n$, so $\dim(\Gamma)=1$. The common level $\mathcal{M}(h,\theta)$ is an $n$-torus, hence our action $J$ coincides (up to a constant factor) with the classical action $\int_{\Gamma} \alpha$ in this case. In the non-integrable case $k=1$, we have $\Gamma=\mathcal{M}(h,\theta)$, so the action is  
$$J(h,\theta)=\int_{H(y;\theta)=h} \alpha\wedge \omega^{n-1} = \int_{H(y;\theta)\leq h} \omega^n,$$
i.e., it equals to the Anosov-Kasuga adiabatic invariant (the volume below the energy level \cite{Ka}).

\begin{theorem} \label{integral_for_the_averaged_system}
Let the frozen system (\ref{Hamiltonian_ODE})
be ergodic with respect to the Liouville measure (\ref{lmmul}) on common levels $\mathcal{M}(h,\theta)$ for a dense set of $(h,\theta)$ in a ball $B\subset \mathbb{R}^k\times \mathbb{R}^1$. Then, for any cobordism class $c$ of closed, codimension-$(k-1)$ submanifolds of
$\mathcal{M}(h,\theta)$, the action $J_c(h,\theta)$ is an
integral of the slow averaged system (\ref{averageslow})
in $B$.
\end{theorem}

\begin{proof}
Let $\overline h(\theta)$ be the solution of the averaged system (\ref{averageslow}), corresponding to 
the initial condition $(h_0, \theta_0)\in B$. We have
\begin{equation}\label{hslow}
\overline h = h_0 + (\theta-\theta_0)\;
\overline{\partial_\theta\mathbf{H}}_{\mathcal{M}} + o(\theta-\theta_0),
\end{equation}
where we denoted by $\overline{\{\;\}}_{\mathcal M}$ the averaging over the Liouville measure $\mu^{_L}_{(h_0,\theta_0)}$ on the manifold $\mathcal{M}:=\mathcal{M}(h_0,\theta_0)$. By
(\ref{hslow}), $\overline h$
stays in a neighbourhood of $h_0$ of size $O(\theta-\theta_0)$, so we remain in $B$ for $|\theta-\theta_0|$ small enough.
        
Our objective is to show that
\[\left. 
\dfrac{d}{d\theta} J_c(\overline h(\theta), \theta)\right|_{\theta = \theta_0} = 0
\]
whenever ergodicity holds on $\mathcal{M}$. Then, since $J_c$ is smooth and the ergodicity holds for a dense set of $(h_0,\theta_0)\in B$, this would give $\dfrac{d}{d\theta} J_c(\overline h(\theta),\theta) = 0$ for all $(h,\theta)\in B$.

By definition
\begin{equation}\label{jc0}
J_c(h_0,\theta_0) = I(\Gamma_0)
\end{equation}
for any smooth, class-$c$ submanifold $\Gamma_0$ of $\mathcal{M}$. 
Take any $t$ and let $\Gamma_1(t;\Gamma_0):= \varphi_{0}^{t}(\Gamma_0)$, where $\varphi_{t_0}^{t}(y_0)$ is the time-dependent flow of the full system (\ref{Hamiltonian_ODE}),(\ref{thetaeq}) with initial condition $(t_0, y_0)$. As one can see from (\ref{slowsub}), $\Gamma_1$ lies in an $O(\varepsilon t)$-neighbourhood of $\mathcal{M}$.

Let $\mathcal{N}$ be a tubular neighborhood of $\mathcal{M}$ in $\mathbb{R}^{2n}$. We coordinatise 
$\mathcal{N}$ as $y=(h, z) \in \mathbb{R}^k \times \mathcal{M}$ where 
$h = \mathbf{H}(y; \theta)$ and $z$ is a smooth projection to $\mathcal{M}$. 
Take a projection map 
$\overline \pi_t (h,z) := (\overline h(\overline\theta(t)),z)$,
where we denote $\overline\theta(t):=\theta_0 + \varepsilon t$,
and define $\Gamma_2(t;\Gamma_0) = \overline\pi_t(\Gamma_1)$. 
By construction, $\overline\pi_t$ is $O(\varepsilon t)$-close to identity, so
$\Gamma_2$ is $O(\varepsilon t)$-close to $\Gamma_1$.
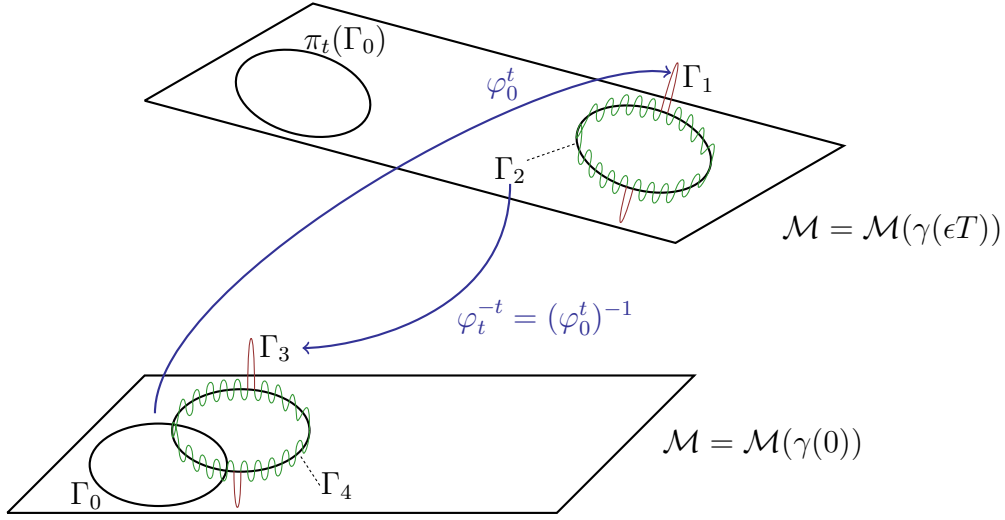
\begin{center}
        \begin{tikzpicture}
                \newcommand{\drawGamma}[0]{
                    \draw[thick] (0ex, 0ex) ellipse (5ex and 3ex);
                }
                \newcommand{\drawGammaone}[0]{
                    \draw[custom_green]
                    (5ex, 0ex)
                        .. controls +(0ex, 1ex) and +(0.5ex, 1ex) ..
                    (4.5ex, 1.3ex)
                        .. controls +(-0.5ex, -1ex) and +(0.1ex, -1ex) ..
                    (4ex, 1.7ex)
                        .. controls +(-0.1ex, 1ex) and +(0.25ex, 1ex) ..
                    (3.5ex, 2ex)
                        .. controls +(-0.25ex, -1ex) and +(0.1ex, -1ex) ..
                    (3ex, 2.4ex)
                        .. controls +(-0.1ex, 1ex) and +(0.1ex, 1ex) ..
                    (2.5ex, 2.7ex)
                        .. controls +(-0.1ex, -1ex) and +(0ex, -1ex) ..
                    (2ex, 2.8ex)
                        .. controls +(0ex, 1ex) and +(0ex, 1ex) ..
                    (1.5ex, 2.9ex)
                        .. controls +(0ex, -1ex) and +(0ex, -1ex) ..
                    (1ex, 2.95ex);
                \draw[custom_red]
                    (1ex, 2.95ex)
                        .. controls +(0ex, 5ex) and +(0ex, 5ex) ..
                    (0.5ex, 2.95ex);
                \draw[custom_green]
                    (0.5ex, 2.95ex)
                        .. controls +(0ex, -1ex) and +(0ex, -1ex) ..
                    (0ex, 3ex)
                        .. controls +(0ex, 1ex) and +(0ex, 1ex) ..
                    (-0.5ex, 2.95ex)
                        .. controls +(0ex, -1ex) and +(0ex, -1ex) ..
                    (-1ex, 2.95ex)
                        .. controls +(0ex, 1ex) and +(0ex, 1ex) ..
                    (-1.5ex, 2.9ex)
                        .. controls +(0ex, -1ex) and +(0ex, -1ex) ..
                    (-2ex, 2.8ex)
                        .. controls +(0.1ex, 1ex) and +(0ex, 1ex) ..
                    (-2.5ex, 2.7ex)
                        .. controls +(0ex, -1ex) and +(0ex, -1ex) ..
                    (-3ex, 2.4ex)
                        .. controls +(0ex, 1ex) and +(0ex, 1ex) ..
                    (-3.5ex, 2ex)
                        .. controls +(0.1ex, -1ex) and +(0.1ex, -1ex) ..
                    (-4ex, 1.7ex)
                        .. controls +(-0.1ex, 1ex) and +(0.1ex, 1ex) ..
                    (-4.5ex, 1.3ex)
                        .. controls +(-0.1ex, -0.5ex) and +(0ex, -1.5ex) ..
                    (-5ex, 0ex)
                        .. controls +(-0ex, 1.5ex) and +(0ex, 1ex) ..
                    (-4.5ex, -1.3ex)
                        .. controls +(0ex, -1ex) and +(0ex, -1ex) ..
                    (-4ex, -1.7ex)
                        .. controls +(0ex, 1ex) and +(0ex, 1ex) ..
                    (-3.5ex, -2ex)
                        .. controls +(0ex, -1ex) and +(0ex, -1ex) ..
                    (-3ex, -2.4ex)
                        .. controls +(0ex, 1ex) and +(0ex, 1ex) ..
                    (-2.5ex, -2.7ex)
                        .. controls +(0ex, -1ex) and +(0ex, -1ex) ..
                    (-2ex, -2.8ex)
                        .. controls +(0ex, 1ex) and +(0ex, 1ex) ..
                    (-1.5ex, -2.9ex)
                        .. controls +(0ex, -1ex) and +(0ex, -1ex) ..
                    (-1ex, -2.95ex)
                        .. controls +(0ex, 1ex) and +(0ex, 1ex) ..
                    (-0.5ex, -2.95ex);
                \draw[custom_red]
                    (-0.5ex, -2.95ex)
                        .. controls +(0ex, -3.5ex) and +(0ex, -3.5ex) ..
                    (0ex, -3ex);
                \draw[custom_green]
                    (0ex, -3ex)
                        .. controls +(0ex, 1ex) and +(0ex, 1ex) ..
                    (0.5ex, -2.95ex)
                        .. controls +(0ex, -1ex) and +(0ex, -1ex) ..
                    (1ex, -2.95ex)
                        .. controls +(0ex, 1ex) and +(0ex, 1ex) ..
                    (1.5ex, -2.9ex)
                        .. controls +(0ex, -1ex) and +(0ex, -1ex) ..
                    (2ex, -2.8ex)
                        .. controls +(0ex, 1ex) and +(0ex, 1ex) ..
                    (2.5ex, -2.7ex)
                        .. controls +(0ex, -1ex) and +(0ex, -1ex) ..
                    (3ex, -2.4ex)
                        .. controls +(0ex, 1ex) and +(0ex, 1ex) ..
                    (3.5ex, -2ex)
                        .. controls +(0ex, -1ex) and +(0ex, -1ex) ..
                    (4ex, -1.7ex)
                        .. controls +(0ex, 1ex) and +(0ex, 1ex) ..
                    (4.5ex, -1.3ex)
                        .. controls +(0ex, -1ex) and +(0ex, -1ex) ..
                    (5ex, 0ex);
                }
                \newcommand{\drawCommonLevel}[0]{
                    \draw[thick] ( 0ex,  0ex) -- (40ex,  0ex) -- (50ex, 10ex) -- (10ex, 10ex) -- ( 0ex,  0ex);
                }
                \drawCommonLevel
                \node at (55ex, 5ex) {$\mathcal{M} = \mathcal{M}(\gamma(0))$};
                    \begin{scope}[shift = {+(11ex, 3.5ex)}]
                        \drawGamma
                        \node (Gamma0) at (-5.25ex, -2.25ex) {$\Gamma_0$};
                    \end{scope}
                    \begin{scope}[shift = {+(17ex, 6ex)}]
                        \drawGamma
                        \node (Gamma4) at (7ex, -4ex) {$\Gamma_4$};
                        \draw[dash pattern=on 1pt off 1pt] (5.75ex, -4ex) -- (4.3ex, -1.8ex);
                        
                        \drawGammaone
                        \node (Gamma3) at (2.5ex, 6ex) {$\Gamma_3$};
                    \end{scope}
                    
                \begin{scope}[shift = {+(10ex, 30ex)}, rotate = -15]
                    \drawCommonLevel
                    \node at (55ex, 5ex) {$\mathcal{M} = \mathcal{M}(\gamma(\epsilon T))$};
                    \begin{scope}[shift = {+(11ex, 3.5ex)}]
                        \drawGamma
                        \node (barGamma0) at (2ex, 4.5ex) {$\pi_t(\Gamma_0)$};
                    \end{scope}
                    \begin{scope}[shift = {+(36ex, 6ex)}]
                        \drawGamma
                        \node (Gamma2) at (-9ex, -4ex) {$\Gamma_2$};
                        \draw[dash pattern=on 1pt off 1pt] (-8ex, -3ex) -- (-4.9ex, -0.9ex);
                        
                        \drawGammaone
                        \node (Gamma1) at (2.5ex, 6ex) {$\Gamma_1$};
                    \end{scope}
                \end{scope}
                \draw[custom_blue, thick, ->]
                    {(Gamma0) + (5ex, 6ex)}
                        .. controls +(0ex, 10ex) and +(-10ex, 2ex) ..
                        node [pos = 0.75, shift = {(-3ex, 1ex)}] (){$\varphi_0^t$}
                    (Gamma1);
                \draw[custom_blue, thick, ->]
                    {(Gamma2) + (0ex, -1ex)}
                        .. controls +(0ex, -10ex) and +(5ex, 0ex) ..
                        node [pos = 0.5, shift = {(9ex, 0ex)}] (){$\varphi_t^{-t} = (\varphi_0^t)^{-1}$}
                    (Gamma3);
        \end{tikzpicture}
\end{center}

By definition
\[J_c(\overline h(\overline\theta(t)), \overline\theta(t)) = I(\overline\pi_t(\Gamma_0)).
\]
Since $\overline \pi_t \circ \varphi_{0}^{s} \ (0 \le s \le t)$ gives an isotopy 
between $\overline \pi_t (\Gamma_0)$ and $\Gamma_2$ in 
$\mathcal{M}((\overline h(\overline\theta(t)), \overline \theta(t)))$, we have
\[J_c(\overline h(\overline \theta(t)), \overline \theta(t)) = I(\Gamma_2),
\]
see (\ref{iiz}).

Let $\Gamma_3(t;\Gamma_0):= \varphi_{t}^{-t}(\Gamma_2)$. For any fixed $t$, we have 
$\Gamma_2\to \Gamma_1$ as $\varepsilon\to 0$, so $\Gamma_3$ is close to $\Gamma_0 = \varphi_{t}^{-t}(\Gamma_1)$, in smooth topology.
Since $\varphi_t^{-t}$ is exact 
symplectic, we have $I(\Gamma_2) = I(\Gamma_3)$, see (\ref{igcon}). Therefore,
\begin{equation}\label{jc3} J_c(\overline h(\overline \theta(t)), \overline \theta(t)) = I(\Gamma_3(t;\Gamma_0))
\end{equation}
for any choice of $\Gamma_0$ in the cobordism class $c$.

\begin{lemma} \label{delta_I_estimate_lemma}
If the flow of the frozen system on $\mathcal{M}$ is ergodic, there exists a constant 
$C > 0$ such that for any $\eta > 0$, there exists $\tau > 0$ such that, 
for an appropriate choice of $\Gamma_0$ in the class $c$,
\begin{align} \label{delta_I_estimate}
|I(\Gamma_3(\tau;\Gamma_0)) - I(\Gamma_0|\; \le \;C \eta \;\varepsilon \tau,
\end{align}
for all sufficiently small $\varepsilon > 0$.
\end{lemma}

This lemma implies the theorem. Indeed, by (\ref{jc0}),(\ref{jc3}), we obtain from 
(\ref{delta_I_estimate})
\[\left|\dfrac{dJ_c}{d\theta}\right|_{\theta = \theta_0} =\;
\left|\lim_{\varepsilon \rightarrow 0} \dfrac{I(\Gamma_3(\tau;\Gamma_0)) - I(\Gamma_0)}{\varepsilon \tau}\right|\le C \eta,
\]
and the theorem follows as $\eta > 0$ can be chosen arbitrarily small.\\

\textit{Proof of Lemma \ref{delta_I_estimate_lemma}.} 
In the coordinates $(h,z)$, system (\ref{Hamiltonian_ODE}),(\ref{thetaeq}) is 
a slow-fast system of the form
\begin{equation}\label{sfs}
\frac{dh}{dt} = \varepsilon F(h,z,\theta_0+\varepsilon t), \qquad
\frac{dz}{dt} = G(h,z,\theta_0+\varepsilon t), 
\end{equation}
where
\begin{equation}\label{fth}
F = \partial_\theta \mathbf{H}.
\end{equation}
Obviously, 
\begin{equation}\label{hh0}
h(t) - h_0 = O(\varepsilon t).
\end{equation}

As $F$ and $G$ are bounded in $C^1$, we have that the derivatives of $(h,z)$
with respect to the initial conditions have at most exponential growth:
\begin{equation}\label{derest}
\partial_{(h_0,z_0)}\; h(t) = O(\varepsilon e^{K_0 |t|}), \qquad 
\partial_{(h_0,z_0)}\; z(t) = O(e^{K_0 |t|}),
\end{equation}
with some constant $K_0$. 

We also have that $z(t)$ depends in a Lipshitz way on the right-hand sides, 
with the Lipshitz constant that grows at most exponentially in $t$.
In particular, we can compare $z(t)$ with the solution $z^f(t; h_0,z_0,\theta_0)$ of the fast subsystem
\begin{equation}\label{fsu0}
\frac{dz}{dt} = G^f(z;h_0,\theta_0) := G(h_0,z,\theta_0) 
\end{equation}
with the same initial conditions: since $G(h(t), z, \theta_0 +\varepsilon t)=G^f(z;h_0,\theta_0) 
+ O(\varepsilon |t|)$ by (\ref{hh0}), we obtain
$$
z(t)=z^f(t) + O(\varepsilon e^{K_0|t|}).
$$
This implies $F(h,z(t),\theta_0+\varepsilon t) = F(h_0, z^f(t), \theta_0) + O(\varepsilon e^{K_0|t|})$,
so integrating the equation for $h$ in (\ref{sfs}) gives
\begin{equation}\label{hhsf}
h(t)=h_0 + \varepsilon \int_0^tF(h_0, z^f(s), \theta_0) ds + O (\varepsilon^2 e^{K_1|t|}),
\end{equation}
with some constant $K_1>K_0$.

Now, take a point $(h_0,z_0)\in \Gamma_0$. Let 
$P^\prime=(h^\prime,z^\prime)=\varphi_0^t(h_0,z_0) \in \Gamma_1$ and consider its projection
$P^{\prime\prime}=(h^{\prime\prime}=\overline h,z^\prime)=\overline\pi_t P^\prime\in \Gamma_2$.
Both $h^\prime$ and $\overline{h}$ are $O(\varepsilon t)$-close to $h_0$, so 
$P^\prime$ and $P^{\prime\prime}$ are $O(\varepsilon t)$-close. Therefore, 
by (\ref{derest}), their backward orbits $\varphi_t^{-s}P^\prime$ and
$\varphi_t^{-s}P^{\prime\prime}$  stay $O(\varepsilon e^{K_1|t|})$ close
for $s\in [0,t]$. Therefore, integrating the equation for $h$ in (\ref{sfs}), we find
$$\tilde h - h_0  = \overline{h} - h^{\prime} + O(\varepsilon^2 e^{K_1|t|}),$$
where we denote $(\tilde h, \tilde z): = \varphi_t^{-t}P^{\prime\prime}\in\Gamma_3$
(recall also that $\varphi_t^{-t}P^\prime=(h_0,z_0)$).

By (\ref{hhsf}),(\ref{fth}),(\ref{hslow}), we conclude 
\begin{equation}\label{mainftr}
\tilde h - h_0 = \varepsilon t \;\left[\;\overline{\partial_\theta \mathbf{H}}_{\mathcal M}\; - \left< \partial_\theta \mathbf{H} \right>_t + o(1)_{\varepsilon\to 0}\right],
\end{equation}
where we use the notation $\left<\;\right>_t$ for the time-average over the solution of the fast system (\ref{fsu0}) on the interval $[0,t]$:
$$\left<\partial_\theta \mathbf{H} \right>_t = \frac{1}{t} \int_0^t \partial_\theta \mathbf{H}|_{h=h_0, z=z^f(s;h_0,z_0,\theta_0),\theta=\theta_0} ds.$$

Consider the smooth manifold $W$ that connects $\Gamma_0: \{h=h_0\}$ and $\Gamma_3: \{h=\tilde h(h_0,z_0), z= \tilde z(h_0,z_0)\}_{(h_0,z_0)\in\Gamma_0}$:
$$W := \bigcup_{z\in \Gamma_0} \{h = h_0 + s (\tilde h(z_0) -h_0), z =
z_0 + s (\tilde z(z_0) -z_0)| \;s\in [0,1]\}, (z_0,h_0)\in \Gamma_0.$$
By Stokes theorem,
$$I(\Gamma_3) - I(\Gamma_0) = \int_{W} \omega^{n-k}.
$$
Because $\left.\omega^{n-k}\right|_{W}$ is a top form, it has a density with respect to the Hausdorff measure on $W$, i.e., the restriction onto $W$ of the Lebesgue volume $(dh)^k (dz)^{2n-k}$. This density is continuous and therefore bounded, uniformly for any manifold $C^1$-close to $W$. Thus, there exists a constant $C_1 > 0$, independent of the choice of $t$ and $\Gamma_0$, such that for all small $\varepsilon > 0$
\begin{equation}\label{wher}\left| I(\Gamma_3) - I(\Gamma_0)\right| \le
C_1 \int_{\Gamma_0} \|\tilde h(h_0,z_0) - h_0\| \;\;\mathcal{H}_{\Gamma_0}(dz_0),
\end{equation}
where the Hausdorff measure $\mathcal{H}_{\Gamma_0}$ is the restriction of the Liouville measure $\mu^{_L}\sim (dz)^{n-k}$ onto $\Gamma_0$.

By the ergodicity of the frozen system on $\mathcal{M}$, for any $\nu>0$ and $\delta>0$, there exists $\tau>0$ such that
$$\|\;\overline{\partial_\theta \mathbf{H}}_{\mathcal{M}} - \left<\partial_\theta \mathbf{H} \right>_\tau\|<\delta$$
for all initial conditions $z_0$ in a neighborhood of $\Gamma_0$ in $\mathcal{M}$ 
outside a set $\mathcal{S}$ of initial conditions such that $\mu^{_L}(\mathcal{S})<\nu$. Since $\mathcal{H}_{\Gamma_0}$ is the restriction of the smooth measure $\mu^{_L}$ 
on $\mathcal{M}$ to the smooth manifold $\Gamma_0$, it is easy to see that one can always make a $C^1$-small
deformation of $\Gamma_0$ to achieve
$$\mathcal{H}_{\Gamma_0} (\Gamma_0\cap \mathcal{S}) < C_2 \nu,$$
for some constant $C_2$ independent of $\nu$ and $\delta$. Then, by (\ref{mainftr}),(\ref{wher}),
$$\left| I(\Gamma_3) - I(\Gamma_0)\right| \le
C (\delta + \nu) \varepsilon \tau$$
for a constant $C$ independent of the choice of $\nu$ and $\delta$.
\end{proof}

Given a cobordism class $c$ of codimension-$(k-1)$ submanifolds of $\mathcal{M}_{h_0,\theta_0}$, the action $J_c$ is a well-defined function of $h$ and $\theta$ in a ball around $(h_0,\theta_0)$. One can extend it by continuity to any region of $(h,\theta)$ for which the integrals 
$\mathbf{H}$ remain independent for all $y\in \mathcal{M}_{h,\theta}$. Thus, $R_c(y,\theta):=J_c(\mathbf{H}(y,\theta),\theta)$ is well-defined for the corresponding region of $(y,\theta)$. Note that if this region is not simply-connected, the function $R$ may be multi-valued, as the cobordism class $c$ can change when making a round along a non-contractible loop in $h$-space.

\begin{theorem}\label{adiabatic_invariant} Let $U$ be a region such that at each $(h,\theta)\in U$ the gradients of the integrals $\mathbf{H}$ are linearly independent at each point of $\mathcal{M}_{h,\theta}$.  If the frozen system (\ref{Hamiltonian_ODE}) on $\mathcal{M}_{h,\theta}$ is ergodic for Lebesgue a.e. $(h,\theta)\in U$, then the action $R_c$ is an adiabatic invariant for the full system (\ref{Hamiltonian_ODE}),(\ref{thetaeq}) in the following sense:
 
For all $T, \delta, \nu > 0$ there exists $\varepsilon_0 > 0$ and a set $\mathcal{B} \subset \mathbf{U}:=\{(y,\theta)| (\mathbf H(y,\theta), \theta)\in U\}$ 
of small Lebesgue measure, i.e., $\mathrm{mes}(B) < \nu$, such that  for all $\varepsilon < \varepsilon_0$ and for any $(y_0,\theta_0) \in \mathbf{U}\setminus \mathbf{B}$
\[ \left|R_c(y(t); \theta_0+\varepsilon t) - R_c(y_0; \theta_0)\right| < \delta,\]
for all $t \in [0, T/\varepsilon]$ such that the solution  $(y(t),\theta_0+\varepsilon t)$ of system (\ref{Hamiltonian_ODE}),(\ref{thetaeq}) with the initial condition $(y_0,\theta_0)$ remains in $\mathbf{U}$.
\end{theorem}
\begin{proof}
The theorem follows from Theorem \ref{integral_for_the_averaged_system} and Anosov ergodic averaging theorem \cite{An}. Namely, by \cite{An}, given a
slow-fast system (e.g., system (\ref{sfs})) with a smooth invariant measure (in our case -- the restriction of the Lebesgue measure onto $\mathcal{M}$), 
if the frozen fast subsystem is ergodic for a.e. values of the slow variavles, then for any $T>0$ the slow component of the solution stays as close as we want to the solution of the averaged system for all initial contions but from a set of measure as small as we want, provided $\varepsilon$ is taken small enough. 
In our case, this means that $h(t) = \mathbf{H}(y(t),\theta_0+\varepsilon t)$ stays close to the solution $\bar h(t)$ of the averaged system (\ref{averageslow}), hence $R_c(y(t), \theta_0+\varepsilon t) = J_c(h(t), \theta_0+\varepsilon t)$ stays close to $J_c(\bar h(t), \theta_0+\varepsilon t)$,
and the latter is constant by  Theorem \ref{integral_for_the_averaged_system}.
\end{proof}

\textbf{Acknowledgements.} The authors are very grateful for support that was available to them. Amir Khodaeian was funded by the Roth Scholarship from Department of Mathematics at Imperial College. Konstantinos Kourliouros was supported by the EPSRC grant EP/Y020669/1.

\end{document}